\documentclass[11pt]{article}

\usepackage[margin=1.3in]{geometry}
\usepackage{amssymb,amsmath,amsfonts, graphicx}
\usepackage{amsthm}
\usepackage{hyperref}
\usepackage{mathrsfs}
\usepackage{textcomp}
\usepackage{enumitem}
\usepackage{supertabular} 
\hypersetup{
    colorlinks,%
    filecolor=black,%
    urlcolor=black
}
\usepackage{tikz}
\usetikzlibrary{matrix,arrows}

\newtheorem{thm}{Theorem}
\newtheorem{cor}[thm]{Corollary}
\newtheorem{lem}[thm]{Lemma}

\newtheorem{ques}{Open Question}

\newtheorem{prop}[thm]{Proposition}
\newtheorem{rem}[thm]{Remark}
\theoremstyle{definition}

\def\XXint#1#2#3{{\setbox0=\hbox{$#1{#2#3}{\int}$}
  \vcenter{\hbox{$#2#3$}}\kern-.5\wd0}}

                \newcommand{\lda}{\lambda}
\newcommand{\om}{\Omega}                \newcommand{\pa}{\partial}
\newcommand{\va}{\varepsilon}           \newcommand{\ud}{\,\mathrm{d}}
\newcommand{\be}{\begin{equation}}      \newcommand{\ee}{\end{equation}}

\newcommand{\R}{\mathbb{R}}

\DeclareMathOperator{\supp}{supp}

\begin{document}

\title{\textbf{On a Rayleigh-Faber-Krahn inequality for the regional fractional Laplacian}}

\author{Tianling Jin\footnote{T. Jin is partially supported by Hong Kong RGC grants ECS 26300716 and GRF 16302519.},\ \  Dennis Kriventsov,\ \  Jingang Xiong\footnote{J. Xiong is partially supported by NSFC 11922104 and 11631002.}}

\date{\today}

\maketitle

\begin{abstract}
We study a Rayleigh-Faber-Krahn inequality for regional fractional Laplacian operators. In particular, we show that there exists a compactly supported nonnegative Sobolev function $u_0$ that attains the infimum (which will be a positive real number) of the set
\[
\left\{ \iint_{\{u > 0\}\times\{u>0\}} \frac{|u(x) - u(y)|^2}{|x - y|^{n + 2 \sigma}}\ud x\ud y : u \in \mathring H^\sigma(\R^n), \int_{\R^n} u^2 = 1, |\{u > 0 \}| \leq 1\right\}.
\]
Unlike the corresponding problem for the usual fractional Laplacian, where the domain of the integration is $\R^n \times \R^n$, symmetrization techniques may not apply here. Our approach is instead based on the direct method and new \emph{a priori} diameter estimates. We also present several remaining open questions concerning the regularity and shape of the minimizers, and the form of the Euler-Lagrange equations.
\end{abstract}

\section{Introduction}

Let $n\ge 1$, $\sigma\in (0,1)$ (with the additional assumption that $\sigma<1/2$ if $n=1$), and $\Omega\subset\R^n$ be an open set. There are two natural fractional Sobolev norms which may be defined for $u \in C^\infty_c(\Omega)$:
$$
I_{n,\sigma,\R^n}[u] := \iint_{\R^n\times\R^n} \frac{(u(x)-u(y))^2}{|x-y|^{n+2\sigma}} \,\ud x\,\ud y
$$
and
$$
I_{n,\sigma,\Omega}[u] := \iint_{\Omega\times\Omega} \frac{(u(x)-u(y))^2}{|x-y|^{n+2\sigma}} \,\ud x\,\ud y.
$$
Depending on the choices of $n,\sigma$ and $\Omega$,  these two norms may or may not be equivalent.  Even when they are equivalent (see Lemma \ref{lem:equivalentnorms}), there are still subtle differences in how they depend on the domain $\Omega$.

One significant difference is the behavior of their corresponding  best Sobolev constants:
\[
S_{n,\sigma}(\Omega):= \inf\left\{I_{n,\sigma,\Omega}[u]: u \in C^\infty_c(\Omega),  \int_{\Omega} |u|^\frac{2n}{n-2\sigma}\,\ud x=1\right\}
\]
and
\[
\widetilde S_{n,\sigma}(\Omega):= \inf\left\{I_{n,\sigma,\R^n}[u]: u \in C^\infty_c(\Omega), \int_{\Omega} |u|^\frac{2n}{n-2\sigma}\,\ud x=1 \right\}\,.
\]
Clearly, $\widetilde S_{n,\sigma}(\Omega) \geq \widetilde S_{n,\sigma}(\R^n)$ and, in fact, using the dilation or translation invariance of $\widetilde S_{n,\sigma}(\R^n)$, it is not difficult to see that
$$
\widetilde S_{n,\sigma}(\Omega) = \widetilde S_{n,\sigma}(\R^n)=S_{n,\sigma}(\R^n) \,.
$$
Moreover, a result of Lieb \cite{Lieb}, classifies all minimizers for $\widetilde S_{n,\sigma}(\R^n)$ and shows that they do not vanish anywhere on $\R^n$. Therefore, the infimum $\widetilde S_{n,\sigma}(\Omega)$ is not attained unless $\Omega=\R^n$.

However, in \cite{FJX}, two of the authors with R. Frank discovered that  the minimization problem for $S_{n,\sigma}(\Omega)$ behaves differently from $\widetilde S_{n,\sigma}(\Omega)$. Let us first recall some qualitative results about  whether the constant $S_{n,\sigma}(\Omega)$ is positive or zero: 
\begin{itemize}
\item For $n\geq 2$ and $\sigma>1/2$, one has $S_{n,\sigma}(\Omega)>0$ for any open set $\Omega$. This follows from Dyda-Frank \cite{DyFr}, which even shows that $\underline{S}_{n,\sigma}:=\inf_{\Omega}S_{n,\sigma}(\Omega)>0$. 

\item When $n\geq 1$ and $\sigma<1/2$, one has $S_{n,\sigma}(\Omega)=0$ for any open set $\Omega$ of finite measure with sufficiently regular boundary; see Lemma 16 in \cite{FJX}. 

\item However, one has $S_{n,\sigma}(\Omega)>0$ for $n\geq 1$ and $\sigma<1/2$ if $\Omega$ is the complement of the closure of a bounded Lipschitz domain or a domain above the graph of a Lipschitz function. This follows from the Sobolev inequality on $\R^n$ and the Hardy inequality from Dyda \cite{Dy}. 
\end{itemize}

More quantitatively speaking, it was shown in \cite{FJX} that the best constant $S_{n,\sigma}(\Omega)$ depends on the domain $\Omega$, and it can be achieved in many cases assuming that $n\ge 4\sigma$:
 
\begin{itemize}
\item If  the complement $\Omega^c$ has an interior point, then  
$
S_{n,\sigma}(\om)<S_{n,\sigma}(\R^n). 
$
\item If $\sigma\neq 1/2$, then $S_{n,\sigma}(\R^n_+)$ is achieved (see also Musina-Nazarov \cite{MN}).     
\item If $\sigma>1/2$, $\Omega$ is a bounded domain such that $B_1^+\subset\Omega\subset \R^n_+$, 
then $
S_{n,\sigma}(\om)<S_{n,\sigma}(\R^n_+).
$
Moreover, if $\partial \Omega $ is smooth then $S_{n,\sigma}(\om)$ is achieved.
\end{itemize}
The discrepancy between the $S_{n,\sigma}(\Omega)$ problem and the $\widetilde S_{n,\sigma}(\Omega)$ problem can be explained as a Br\'ezis-Nirenberg \cite{BrNi}  effect:
\begin{align*}
I_{n,\sigma,\Omega}[u] &= I_{n,\sigma,\R^n}[u] - 2\int_{\Omega}u^2(x)\,\ud x\int_{\R^n\setminus\Omega}\frac{1}{|x-y|^{n+2\sigma}} \,\ud y\\
& \approx I_{n,\sigma,\R^n}[u] - c_{n,\sigma}\int_{\Omega}\frac{u^2(x)}{\mbox{dist}(x,\partial\Omega)^{2\sigma}}\,\ud x\quad\forall\,u\in C_c^\infty(\Omega).
\end{align*}
Therefore, the $S_{n,\sigma}(\Omega)$ problem is the $\widetilde S_{n,\sigma}(\Omega)$ problem with an additional negative term, and  it is this term that for $n\ge4\sigma$ lowers the value of the infimum and produces a minimizer. This fact was first observed by  Br\'ezis-Nirenberg \cite{BrNi} in the Laplacian setting.

In probability, $I_{n,\sigma,\Omega}$ is called the Dirichlet form of the censored $2\sigma$-stable process \cite{BBC} in $\Omega$. Its generator 
\begin{equation}\label{eq:regional}
(-\Delta)_{\Omega}^{\sigma}u:=2 \lim_{\va\to 0} \int_{\{y\in\Omega:\ |y-x|\ge \va\}} \frac{u(x)-u(y)}{|x-y|^{n+2\sigma}}\,\ud y
\end{equation}
is usually called the regional fractional Laplacian operator \cite{Guan, GuanMa}. Therefore, the difference between $S_{n,\sigma}(\Omega)$ and $\widetilde S_{n,\sigma}(\Omega)$ is also related to the difference between this regional fractional Laplacian and the ``full'' fractional Laplacian on $\R^n$, and in turn by the nonlocal Hardy-type term's dependence on $\Omega$.

The nontrivial dependence of $S_{n, \sigma}(\Omega)$ on $\Omega$ leads to interesting and natural questions of a shape-optimization nature. For example: it was asked in \cite{FJX} that
\begin{ques}
	Assume $n \geq 2$ and $\sigma > \frac{1}{2}$. What are all the open sets $\Omega \subset \R^n$ with $S_{n, \sigma}(\Omega) = \underline{S}_{n, \sigma}$, where, as mentioned before, $\underline{S}_{n, \sigma}=\inf\{S_{n,\sigma}(V): V\subset\R^n\mbox{ is an open set}\}>0$?
\end{ques}
As far as we are aware, it is unknown whether there are any such sets. A tempting conjecture would be that all such sets are balls, but the nature of the underlying equations does not support the arguments usually used to prove such symmetry results (symmetrization, moving planes; see the discussion below).

We do not attack this question here: there are several difficulties in applying calculus of variations, shape optimization, and free boundary techniques to it, and we wish to first focus on the challenges related to the nonlocality of the equation and the subtle dependence of $I_{n, \sigma, \Omega}$ on $\Omega$ in isolation. To do so, we formulate an optimization problem with similar characteristics but which avoids the criticality and extra scale invariance associated with the optimal Sobolev exponent $\frac{2n}{n - 2 \sigma}$.

Let $n\ge 2$, $\Omega\subset\R^n$ be a bounded open set,  $\sigma\in (\frac 12, 1)$ and $\mathring H^\sigma(\Omega)$ be the completion of $C^1_c(\Omega)$ with respect to the quadratic form $I_{n,\sigma,\Omega}$. Then we have the compact fractional Sobolev embedding $\mathring H^\sigma(\Omega)\hookrightarrow L^2(\Omega)$. Therefore, the eigenvalues defined as 
\begin{equation}\label{eq:eigen}
\begin{cases}
(-\Delta)_\Omega^{\sigma}u=\lambda u\quad\mbox{in }\Omega\\
u=0\quad\mbox{on }\partial\Omega
\end{cases}
\end{equation}
consist a sequence that can be ordered (counting the multiplicities) as
\[
0<\lda_{1,\sigma}(\Omega)\le \lda_{2,\sigma}(\Omega)\le\cdots\le\lda_{k,\sigma}(\Omega)\le\cdots\to\infty.
\]
Moreover, we have for the first eigenvalue that
\[
\lda_{1,\sigma}(\Omega)=\min\{I_{n,\sigma,\Omega}[u]: u\in \mathring H^\sigma(\Omega), \|u\|_{L^2(\Omega)}=1\}.
\]
Let $u\in \mathring H^\sigma(\Omega)$ be an eigenfunction for $\lda_{1,\sigma}(\Omega)$. Then by the fractional Sobolev inequality and H\"older inequality, we have
\[
\lda_{1,\sigma}(\Omega)\|u\|^2_{L^2(\Omega)}=I_{n,\sigma,\Omega}[u]\ge S_{n,\sigma}(\Omega)\|u\|^2_{L^{\frac{2n}{n-2\sigma}}(\Omega)}\ge S_{n,\sigma}(\Omega) |\Omega|^{-\frac{2\sigma}{n}}\|u\|^2_{L^2(\Omega)}.
\]
Therefore,
\[
\lda_{1,\sigma}(\Omega)\ge S_{n,\sigma}(\Omega) |\Omega|^{-\frac{2\sigma}{n}}\ge \underline S_{n,\sigma} \cdot |\Omega|^{-\frac{2\sigma}{n}}.
\]

Inspired by the classical Rayleigh-Faber-Krahn inequality \cite{Faber, Krahn, Rayleigh}, we would like to study the variational problem
\begin{equation}\label{eq:fractionalRFK}
\inf\{\lda_{1,\sigma}(\Omega): \ \Omega\subset\R^n \mbox{ a bounded open set such that } |\Omega|=|B_1|\},
\end{equation}
where $|\Omega|$ is the Lebesgue measure of $\Omega$.

One notable difficulty of this variational problem is that radial symmetrization rearrangement may \emph{not} work. Recall that for $u\in C^1_c(\Omega)$, the P\'olya-Szeg\"o inequality \cite{PS} states that 
\begin{equation}\label{eq:PS}
\int_{\Omega}|\nabla u|^p\ge \int_{\Omega^*}|\nabla u^*|^p
\end{equation}
for $1\le p<\infty$, where $\Omega^*$ is the symmetric rearrangement of the set $\Omega$ and $u^*$ is the symmetric decreasing rearrangement of the function $u$ (for the definitions of rearrangements we refer to the book of Lieb-Loss \cite{LiebLoss}). Meanwhile, P\'olya-Szeg\"o type inequality also holds for $I_{n,\sigma,\R^n}[u]$. Namely, for $0<\sigma<1$ and $1\le p<\infty$, it was shown in Theorem 9.2 in Almgren-Lieb \cite{AL} that for every $u\in C^1_c(\R^n)$
\begin{equation}\label{eq:PS2}
{\iint_{\R^n\times \R^n} \frac{|u(x)-u(y)|^p}{|x-y|^{n+\sigma p}}\,\ud x\ud y}\ge {\iint_{\R^n\times \R^n} \frac{|u^*(x)-u^*(y)|^p}{|x-y|^{n+\sigma p}}\,\ud x\ud y}. 
\end{equation}
However, in the interesting paper \cite{LW}, it was showed that there exists a nonnegative $u\in C^\infty_c(B_1)$ such that
\begin{equation}\label{eq:sym}
\iint_{B_1 \times B_1} \frac{|u(x)-u(y)|^p}{|x-y|^{n+\sigma p}}\,\ud x\ud y
< \iint_{B_1\times B_1} \frac{|u^*(x)-u^*(y)|^p}{|x-y|^{n+\sigma p}}\,\ud x\ud y.
\end{equation}
This failure of P\'olya-Szeg\"o type inequality for $I_{n,\sigma,\Omega}[u]$ is another notable difference between the two norms $I_{n,\sigma,\R^n}[u]$ and $I_{n,\sigma,\Omega}[u]$ on the set $C_c^\infty(\Omega)$.

If one considers the first eigenvalue of the classical fractional Laplacian
\begin{equation}\label{eq:eigen2}
\begin{cases}
(-\Delta)^{\sigma}u=\widetilde\lambda_{1,\sigma}(\Omega) u\quad\mbox{in }\Omega\\
u=0\quad\mbox{on }\R^n\setminus\Omega,
\end{cases}
\end{equation}
then, using the P\'olya-Szeg\"o type inequality \eqref{eq:PS2}  for $I_{n,\sigma,\R^n}[u]$, we know that $\widetilde\lambda_{1,\sigma}(\Omega)\ge \widetilde\lambda_{1,\sigma}(B_1)$ for all open set $\Omega$ of the same measure as $B_1$, see Sire-V\'azquez-Volzone \cite{SVV}. 

To remove the measure constraint in \eqref{eq:fractionalRFK}, one can consider the following equivalent problem (up to scaling): 
\begin{equation}\label{prob:variational}
\inf\{\lda_{1,\sigma}(\Omega)+|\Omega|: \ \Omega\subset\R^n \mbox{ is a bounded open set}\}.
\end{equation}
As we will show later in Proposition \ref{prop:equivalentproblems}, 
\begin{align*}
&\inf\{\lda_{1,\sigma}(\Omega)+|\Omega|: \ \Omega\subset\R^n \mbox{ is a bounded open set}\}\\
&= \inf\left\{\frac{I_{n,\sigma,\{u>0\}}[u] }{\|u\|^2_{L^2(\R^n)}}+|\{u>0\}|: u\in \mathring H^\sigma(\R^n),\ \ u\not\equiv 0,\ u\ge 0\mbox{ in }\R^n\right\}.
\end{align*}

In this paper, we prove the following existence result:
\begin{thm}\label{thm:existenceg}
Let $n\ge 2$ and $\sigma\in (\frac 12,1)$. There exists $u_0\in \mathring H^\sigma(\R^n)$, $u_0\not\equiv 0$, $u_0\ge 0\mbox{ in }\R^n$ such that $\{u_0>0\}$ is bounded, and 
\begin{equation}\label{eq:quantity}
\inf\left\{\frac{I_{n,\sigma,\{u>0\}}[u] }{\|u\|^2_{L^2(\R^n)}}+|\{u>0\}|: u\in \mathring H^\sigma(\R^n),\ \ u\not\equiv 0,\ u\ge 0\mbox{ in }\R^n\right\}
\end{equation}
is achieved by $u_0$.
\end{thm}
The key difficulty in proving this is to obtain an \emph{a priori} bound on the diameter of $\{u_0 > 0\}$. In the shape optimization and free boundary literature, such bounds are common and are closely tied with lower bounds on the growth of $u_0$ from the boundary of $\partial \{u_0 > 0 \}$, or perhaps of auxiliary functions related to $u_0$: the idea is that if $u_0$ grows at a prescribed rate from $\partial \{u_0 > 0 \}$, its support cannot have long, thin necks or many small connected components. Such growth estimates are usually obtained by using sets like $\{u_0 > 0\} \setminus B_R(x)$ as competitors (with appropriately chosen functions). We proceed along these lines here as well, but carrying out the argument requires various non-standard modifications, some decidedly nonlocal multi-scale iteration procedures, and in the end does not result in ``uniform'' growth estimates on $u_0$.

For related reasons, we are unable to show much more than stated in Theorem \ref{thm:existenceg} concerning the nature of the minimizing $u_0$ (though see Lemma \ref{lem:upperbound}, which shows $u_0$ is bounded). Of particular interest is this question:
\begin{ques}\label{ques:continuity}
Is \eqref{eq:quantity} achieved by a continuous function $u_0\in \mathring H^\sigma(\R^n)$?
\end{ques}
An affirmative answer would imply that problems \eqref{eq:quantity} and \eqref{prob:variational} have identical minimizers (in the sense that the support of a minimizer of \eqref{eq:quantity} is a minimizer of \eqref{prob:variational}, and vice versa). The difficulty of proving this continuity is in estimating the energy difference between $u_0$ and its competitor whose support is slightly enlarged. See Section \ref{sec:discussion} for further discussion and additional open problems.

If one restricts the problem \eqref{eq:fractionalRFK} to the class of convex sets, then we have
\begin{thm}\label{thm:convex}
Let $n\ge 2$ and $\sigma\in (\frac 12,1)$.  There exists a convex open set that achieves 
\[
\inf\{\lda_{1,\sigma}(\Omega): \Omega\subset\R^n \mbox{ is a convex open set  and } |\Omega|=1\}.
\]
\end{thm}

This paper is organized as follows. In Section \ref{sec:inequality}, we recall some inequalities that are needed here. In Section \ref{sec:convex}, we prove the existence of minimizers for the convex case: this is much simpler, but illustrates the concepts in play. In Section \ref{sec:functionformula}, we reformulate the variational problem \eqref{prob:variational} as \eqref{eq:quantity}. In Section \ref{sec:upperbound}, we show a local pointwise upper bound on the minimizer.  In Section \ref{sec:lowerbound}, we show a lower bound of the minimizer, use it to estimate the diameter of the support of $u_0$, and conclude the proof of Theorem \ref{thm:existenceg}. In the last section, we discuss Open Question \ref{ques:continuity} and related topics.

\bigskip

\noindent \textbf{Acknowledgements:} Part of this work was completed while the second and third named authors were visiting the Hong Kong University of Science and Technology, to which they are grateful for providing  very stimulating research environments and supports.

\section{Some inequalities}\label{sec:inequality}
Let us first recall a sharp Hardy inequality for fractional integrals on general domains. Let $\Omega\subsetneq\R^n $ be an open set. For a direction $\omega\in\mathbb{S}^{n-1}$, we define
\[
d_{\omega,\Omega}(x)=\inf\{|t|: x+t\omega\not\in\Omega\}
\]
and for $\alpha>1$ that
\[
m_{\alpha}(x)=\left(\frac{2\pi^{\frac{n-1}{2}} \Gamma(\frac{1+\alpha}{2})}{\Gamma(\frac{N+\alpha}{2})}\right)^{\frac{1}{\alpha}} \left(\int_{\mathbb{S}^{n-1}}\frac{1}{d_{\omega,\Omega}(x)^\alpha}\,\ud \omega\right)^{-\frac{1}{\alpha}}.
\]
\begin{thm}[Theorem 1.2 in Loss-Sloane \cite{LS}]\label{thm:fractionalHardy}
Let $\frac 12<\sigma<\frac p2<\infty$ and $n\ge 1$. Then for any $u\in C^\infty_c(\Omega)$, 
\[
\iint_{\Omega\times\Omega} \frac{(u(x)-u(y))^p}{|x-y|^{n+2\sigma}} \,\ud x\,\ud y\ge C_{n,p,\sigma}\int_{\Omega}\frac{|u(x)|^p}{m_{2\sigma}(x)^{2\sigma}}\,\ud x,
\]
where $C_{n,p,\sigma}=2\pi^{\frac{n-1}{2}} \frac{\Gamma(\frac{1+2\sigma}{2})}{\Gamma(\frac{N+2\sigma}{2})}\int_0^1\frac{|1-r^{\frac{2\sigma-1}{p}}|^p}{(1-r)^{1+2\sigma}}\,\ud r$ is sharp. 
\end{thm}
Consequently, we can show that  $I_{n,\sigma,\R^n}[u]$ can be controlled by $I_{n,\sigma,\Omega}[u]$. 
\begin{lem}\label{lem:equivalentnorms}
Let $n\ge 2$ and $\sigma\in (1/2,1)$. There exists a constant $C=C(n,\sigma)>0$ such that for all open sets $\Omega\subset\R^n$ and all $u\in \mathring H^\sigma(\Omega)$,
\begin{equation}\label{eq:equivalent}
\iint_{\R^n\times\R^n} \frac{(u(x)-u(y))^2}{|x-y|^{n+2\sigma}} \,\ud x\ud y\le C \iint_{\Omega\times\Omega} \frac{(u(x)-u(y))^2}{|x-y|^{n+2\sigma}} \,\ud x\ud y.
\end{equation}
\end{lem}
\begin{proof}
By a density argument, we only need to show this for $u\in C_c^1(\Omega)$. Suppose $\Omega\subsetneq\R^n $.
We have
\begin{align}
&\iint_{\R^n\times \R^n} \frac{|u(x)-u(y)|^2}{|x-y|^{n+2\sigma }}\,\ud x\ud y \nonumber\\
&=\iint_{\Omega\times \Omega} \frac{|u(x)-u(y)|^2}{|x-y|^{n+ 2\sigma }}\,\ud x\ud y+2\int_{\Omega}u^2(x)\left(\int_{\R^n\setminus \Omega}\frac{1}{|x-y|^{n+2\sigma }}\,\ud y\right)\ud x. \label{eq:aux2}
\end{align}
Also, for $x\in\Omega$, we have
\begin{align*}
\int_{\R^n\setminus \Omega}\frac{1}{|x-y|^{n+2\sigma }}\,\ud y\le \int_{\mathbb{S}^{n-1}}\ud \omega \int_{d_{\omega,\Omega}(x)}^{\infty}\frac{1}{r^{1+2\sigma}}\,\ud r&=\int_{\mathbb{S}^{n-1}} \frac{1}{2\sigma d_{\omega,\Omega}(x)^{2\sigma}}\ud \omega=\frac{C(n,\sigma)}{(m_{2\sigma}(x))^{2\sigma }}
\end{align*}
for some constant $C(n,\sigma)$ depending only on $n,\sigma$, but \emph{not} on $\Omega$. Thus, we have
\begin{align}
\int_{\Omega}u^2(x)\left(\int_{\R^n\setminus \Omega}\frac{1}{|x-y|^{n+2\sigma}}\,\ud y\right)\ud x
&\le C(n,\sigma) \int_{\Omega}\frac{u^2(x)}{(m_{2\sigma}(x))^{2\sigma}} \,\ud x\nonumber\\
&\le C(n,\sigma) \iint_{\Omega\times \Omega} \frac{|u(x)-u(y)|^2}{|x-y|^{n+2 \sigma}}\,\ud x\ud y,\label{eq:auxsobolev}
\end{align}
where we used Theorem \ref{thm:fractionalHardy} in the last inequality. 
\end{proof}

We rewrite Lemma \ref{lem:equivalentnorms} into another form for convenience. 
\begin{lem}\label{lem:sobolevforus}
There exists $C=C(n,\sigma)$ such that
\[
\iint_{\R^n\times \R^n}\frac{(u(x)-u(y))^2}{|x-y|^{n+2\sigma}}\,\ud x\ud y
\le C \iint_{\{u>0\}\times \{u>0\}}\frac{(u(x)-u(y))^2}{|x-y|^{n+2\sigma}}\,\ud x\ud y
\]
for every nonnegative function $u\in \mathring H^\sigma(\R^n)$.
\end{lem}
\begin{proof}
Without loss of generality, we assume that $|\R^n\setminus\{u>0\}|>0$. 
We also assume that $\iint_{\{u>0\}\times \{u>0\}}\frac{(u(x)-u(y))^2}{|x-y|^{n+2\sigma}}\,\ud x\ud y=1$. Using only the outer regularity of Lebesgue measure, we can find a sequence of open sets $\Omega_u^k\subset\R^n$ such that $\{u>0\}\subset\Omega_u^k \subset \Omega_u^{k-1}$ and $|\Omega_u^k\setminus\{u>0\}| \rightarrow 0$. Since $u\in \mathring H^\sigma(\R^n)$, the monotone convergence theorem gives that
\[
\lim_{k\to\infty} \iint_{\Omega_u^k\times\Omega_u^k }\frac{(u(x)-u(y))^2}{|x-y|^{n+2\sigma}}\,\ud x\ud y = \iint_{\{u > 0\}\times \{u > 0\} }\frac{(u(x)-u(y))^2}{|x-y|^{n+2\sigma}}\,\ud x\ud y = 1.
\]
By \eqref{eq:equivalent} applied to each $\Omega_u^k$, 
\[
\iint_{\R^n\times\R^n }\frac{(u(x)-u(y))^2}{|x-y|^{n+2\sigma}}\,\ud x\ud y\le C(n,\sigma) \inf_k \iint_{\Omega_u^k\times\Omega_u^k }\frac{(u(x)-u(y))^2}{|x-y|^{n+2\sigma}} = C(n, \sigma),
\]
giving the conclusion.
\end{proof}

\section{Convex domains}\label{sec:convex}

We first consider the special case of the minimizing problem among all convex sets.

\begin{proof}[Proof of Theorem \ref{thm:convex}]

We will adapt the proof in Lin \cite{Lin} for the Laplacian case. 

First, because of \eqref{eq:equivalent}, we have
\begin{equation}\label{eigenequi}
\widetilde\lambda_{1,\sigma}(\Omega)\le C \lambda_{1,\sigma}(\Omega),
\end{equation}
where $\widetilde\lambda_{1,\sigma}$ is defined in \eqref{eq:eigen2}.

Let $\{\Omega_k\}$ be a minimizing sequence. By John's Lemma, we have that either (A) $\Omega_k$ converges in the Hausdorff distance to a bounded convex  set  $\Omega_\infty$ with $|\Omega_\infty|=1$ (we may assume that $\Omega_\infty$ is open since the boundary of $\Omega_\infty$ is of measure zero), or (B) there is a subsequence, still denoted $\Omega_k$, such that $\Omega_k$ are contained in strips (after suitable rotations and translations, noting that $\lda_{1,\sigma}(\Omega)$ is invariant under them) of form $[-\delta_k,\delta_k]\times [-L_k,L_k]^{n-1}$ such that $\delta_k\to 0^+$ and $L_k\to\infty$. 

We will show that in case (B), $\widetilde\lambda_{1,\sigma}(\Omega_k)\to\infty$, which contradicts \eqref{eigenequi} and the fact that $\Omega_k$ is a minimizing sequence. Denote $Q_k=[-\delta_k,\delta_k]\times [-L_k,L_k]^{n-1}$. Since $C^\infty_c(\Omega_k)\subset C^\infty_c(Q_k)$, it is clear that $\widetilde\lambda_{1,\sigma}(Q_k)\le \widetilde\lambda_{1,\sigma}(\Omega_k)$. Therefore, we only need to show that $\widetilde\lambda_{1,\sigma}(Q_k)\to\infty$. 

Let $Q_0=[-1,1]\times \R^{n-1}$. We claim that there exists $C>0$ such that
\begin{equation}\label{eq:globaleigen}
\|\varphi\|_{L^2(\R^n)}\le C \|\varphi\|_{\mathring H^\sigma(\R^n)}\quad\mbox{for all }\varphi\in C^\infty_c(Q_0).
\end{equation}
Indeed, for every $x'\in\R^{n-1}$, we have
\[
\int_{-1}^1 \varphi(x_1,x')^2\,\ud x_1\le C \|\varphi(\cdot,x')\|^2_{\mathring H^\sigma(\R^1)}= C \int_\R |\xi_1|^{2\sigma}\phi(\xi_1,x')^2\,\ud \xi_1,
\]
where $\phi(\xi_1,x')$ is the (full) Fourier transform of $\varphi(x_1,x')$ in the $x_1$ variable. Integrating the above inequality, we have
\begin{align*}
\int_{\R^n} \varphi(x)^2\,\ud x&\le  C\int_\R |\xi_1|^{2\sigma} \,\ud \xi_1\int_{\R^{n-1}}\phi(\xi_1,x')^2\,\ud x'\\
&= C\int_\R |\xi_1|^{2\sigma} \,\ud \xi_1\int_{\R^{n-1}}\hat \varphi(\xi_1,\xi')^2\,\ud \xi'\\
&= C\int_{\R^n} |\xi_1|^{2\sigma} \hat \varphi(\xi)^2\,\ud \xi\\
&\le C\int_{\R^n} |\xi|^{2\sigma} \hat \varphi(\xi)^2\,\ud \xi=C \|\varphi\|_{\mathring H^\sigma(\R^n)},
\end{align*}
where $\hat\varphi$ is the Fourier transform of $\varphi$, and we used the Plancherel theorem in the first equality. This proves \eqref{eq:globaleigen}. 

Then we have
\begin{align*}
\widetilde\lambda_{1,\sigma}(Q_k)&\ge \inf_{\varphi\in C^\infty_c([-\delta_k,\delta_k]\times\R^{n-1})}\frac{\|\varphi\|^2_{\mathring H^\sigma(\R^n)}}{\|\varphi\|^2_{L^2(\R^n)}}\\
&= \frac{1}{\delta_k^{2\sigma}} \inf_{\varphi\in C^\infty_c([-1,1]\times\R^{n-1})}\frac{\|\varphi\|^2_{\mathring H^\sigma(\R^n)}}{\|\varphi\|^2_{L^2(\R^n)}}\\
&\ge \frac{1}{C\delta_k^{2\sigma}}\\
&\to \infty\quad\mbox{as }k\to\infty.
\end{align*}
We reached a contradiction, and thus case (A) holds: $\Omega_k$ converges in the Hausdorff distance to a bounded convex open set  $\Omega_\infty$ with $|\Omega_\infty|=1$.

Suppose that 
\[
\lda_{1,\sigma}(\Omega_k)=I_{n,\sigma,\Omega_k}[u_k]
\]
for some $u_k\in \mathring H^\sigma(\Omega_k)$ such that $\|u_k\|_{L^2(\Omega_k)}=1$. We extend $u_k$ to be zero in $\R^n\setminus\Omega_k$. By \eqref{eq:equivalent}, we know that $u_k$ is bounded in $\mathring H^\sigma(\R^n)$. Then subject to a subsequence,  there exists $u\in \mathring H^\sigma(\R^n)$ such that $u_k\rightharpoonup u$ weakly in $\mathring H^\sigma(\R^n)$ and $u_k\to u$ in $L^2(\R^n)$. Hence, $u\equiv 0$ in $\R^n\setminus\Omega_\infty$ and $\|u\|_{L^2(\Omega_\infty)}=1$. Morevoer, by Fatou's Lemma, we have
\begin{align*}
\liminf_{k\to\infty}\iint_{\Omega_k\times\Omega_k} \frac{(u_k(x)-u_k(y))^2}{|x-y|^{n+2\sigma}} \,\ud x\,\ud y\ge \iint_{\Omega_\infty\times\Omega_\infty} \frac{(u(x)-u(y))^2}{|x-y|^{n+2\sigma}} \,\ud x\,\ud y.
\end{align*}
Hence,
\[
\lda_{1,\sigma}(\Omega_\infty)=\inf\{\lda_{1,\sigma}(\Omega): \Omega\subset\R^n \mbox{ is a convex open set and } |\Omega|=1\}.
\]
This finishes the proof of this theorem.
\end{proof}

\section{A formulation for functions}\label{sec:functionformula}
In this section, we would like reformulate the  variational problem \eqref{prob:variational} for domains to  a variational problem for functions. Since $||u(x)|-|u(y)||\le |u(x)-u(y)|$, we know that the first eigenfunction, that are the solutions of \eqref{eq:eigen}, do not change signs in $\Omega$.  The next lemma states that they do not vanish in $\Omega$.

\begin{lem}\label{lem:positiveeigenfunction}
Let $n\ge 2$, $\Omega\subset\R^n$ be a open set,  $\sigma\in (\frac 12, 1)$. Let $ u\in \mathring H^\sigma(\Omega), u\not\equiv 0,$ be a nonnegative weak solution of \eqref{eq:eigen}, that is, 
\[
\iint_{\Omega\times\Omega}\frac{(u(x)-u(y))(\varphi(x)-\varphi(y))}{|x-y|^{n+2\sigma}}\,\ud x\ud y=\lambda_{1,\Omega}\int_{\Omega}u(x)\varphi(x)\,\ud x
\]
for every $\varphi\in \mathring H^\sigma(\Omega)$. Then $u$ is smooth and positive in $\Omega$.
\end{lem}

\begin{proof}
For $u,\varphi\in\mathring H^\sigma(\Omega)$, we extend $u$ and $\varphi$ to be identically zero in $\R^n\setminus\Omega$, and by Lemma \ref{lem:equivalentnorms}, they are functions in $\mathring H^\sigma(\R^n)$. We can rewrite the integral as
\[
\iint_{\R^n\times\R^n}\frac{(u(x)-u(y))(\varphi(x)-\varphi(y))}{|x-y|^{n+2\sigma}}\,\ud x\ud y-\int_{\Omega}c(x)u(x)\varphi(x)\,\ud x =\lambda_{1,\Omega}\int_{\Omega}u(x)\varphi(x)\,\ud x,
\]
where
\[
c(x)=\int_{\R^n\setminus\Omega} \frac{1}{|x-y|^{n+2\sigma}}\,\ud y.
\]
That is, $u$ is a weak solution of
\begin{align*}
(-\Delta)^\sigma u(x)-c(x)u(x)&=\lambda_{1,\Omega} u(x)\quad\mbox{on }\Omega,\\
u&=0\quad\mbox{in }\R^n\setminus\Omega.
\end{align*}
Since $c(x)$ is smooth in $\Omega$, by the standard regularity theory for fractional Laplacian equations, $u$ is smooth in $\Omega$. Hence, the equation \eqref{eq:eigen} holds pointwise, that is, 
\[
2 \lim_{\va\to 0} \int_{\{y\in\Omega:\ |y-x|\ge \va\}} \frac{u(x)-u(y)}{|x-y|^{n+2\sigma}}\,\ud y=\lambda_{1,\Omega} u(x)\quad\mbox{for every }x\in\Omega.
\]
Therefore, if there exists $x_0\in\Omega$ such that $u(x_0)=0$, then since $u$ is nonnegative in $\Omega$, $u$ must be identically zero. This is a contradiction.
\end{proof}

Let $N>0$ be a positive real number. Let us consider an alternative variational problem
\begin{equation}\label{eq:equivalentvp}
m(N):=\inf\left\{\frac{I_{n,\sigma,\{u>0\}}[u] }{\|u\|^2_{L^2(B_N)}}+|\{u>0\}|: u\in \mathring H^\sigma(B_N),\ \ u\not\equiv 0, u\ge 0\mbox{ in }B_N\right\}.
\end{equation}

\begin{lem}\label{lem:equiv} $m(N)=\inf\left\{\lda_{1,\sigma}(\Omega)+|\Omega|: \Omega\subset B_N \mbox{ is an open set}  \right\}.$
\end{lem}

\begin{proof}
Let $\Omega\subset B_N$ be an open set, and let $u\in \mathring H^\sigma(\Omega)$ be such that $\lda_{1,\sigma}(\Omega)= I_{n,\sigma,\Omega}[u]$, $\|u\|_{L^2}=1$. Then, by Lemma \ref{lem:positiveeigenfunction}, we can choose that $u>0$ in $\Omega$.   Hence,
\[
m(N)\le \lda_{1,\sigma}(\Omega)+|\Omega|.
\]
Taking the infimum over all open sets, we obtain
\[
m(N)\le \inf\left\{\lda_{1,\sigma}(\Omega)+|\Omega|: \Omega\subset B_N \mbox{ is an open set}  \right\}.
\]
On the other hand, let $u\in \mathring H^\sigma(B_N)$ with $u\ge 0$  in $B_N$. For every $\va>0$, there exists an open set $\Omega_\va\subset B_N$ (otherwise just considering $\Omega_\va\cap B_N$) such that  $\{u>0\}\subset\Omega_\va$, $|\Omega_\va\setminus\{u>0\}|$ is small, and 
\[
\iint_{\Omega_\va\times\Omega_\va}\frac{(u(x)-u(y))^2}{|x-y|^{n+2\sigma}}\,\ud x\ud y\le \iint_{\{u>0\}\times\{u>0\}}\frac{(u(x)-u(y))^2}{|x-y|^{n+2\sigma}}\,\ud x\ud y+\va.
\]
Moreover, there exists a sequence of functions $u_j\in C^\infty_c(\Omega_\va)$ with $u_j \geq 0$ such that $u_j \to u$ in $H^\sigma(\Omega_\va)$. Since $\{u_j>0\}$ is open, we have
\[
\frac{I_{n,\sigma,\{u_j>0\}}[u_j] }{\|u_j\|^2_{L^2(B_N)}}+|\{u_j>0\}|\ge \inf\left\{\lda_{1,\sigma}(\Omega)+|\Omega|: \Omega\subset B_N \mbox{ is an open set}  \right\}.
\]
Since
\[
\frac{I_{n,\sigma,\{u_j>0\}}[u_j] }{\|u_j\|^2_{L^2(B_N)}}\le \frac{I_{n,\sigma,\Omega_\va}[u_j] }{\|u_j\|^2_{L^2(B_N)}}\le  \frac{I_{n,\sigma,\Omega_\va}[u]+\va}{\|u_j\|^2_{L^2(B_N)}}\le  \frac{I_{n,\sigma,\{u>0\}}[u] +2\va}{\|u_j\|^2_{L^2(B_N)}}
\]
for all large $j$, we can send $j\to\infty$ and then $\va\to0$ to obtain
\[
m(N)\ge \inf\left\{\lda_{1,\sigma}(\Omega)+|\Omega|: \Omega\subset B_N \mbox{ is an open set}  \right\}.
\]
\end{proof}
\begin{lem}\label{lem:equivachieved}
The $m(N)$ defined in \eqref{eq:equivalentvp} is attained by a function $u\in\mathring H^\sigma(B_N)$.
\end{lem}
\begin{proof}
Let $u_k$ be a minimizing sequence of $m$ in \eqref{eq:equivalentvp}. Suppose that $\|u_k\|_{L^2(B_N)}=1$. Since $u_k\in \mathring H^\sigma(B_N)$, we know from \eqref{eq:equivalent}  that $u_k\in\mathring H^\sigma(\R^n)$, and 
\[
\iint_{\{u_k>0\}\times\{u_k>0\}} \frac{(u_k(x)-u_k(y))^2}{|x-y|^{n+2\sigma}} \,\ud x\,\ud y\le C.
\]
By Lemma \ref{lem:sobolevforus}, we know that $\{u_k\}$ is a bounded sequence in $\mathring H^\sigma(\R^n)$.  Then passing to a subsequence,  there exists $u\in \mathring H^\sigma(\R^n)$ such that $u_k\rightharpoonup u$ weakly in $\mathring H^\sigma(\R^n)$ and $u_k\to u$ strongly in $L^2(B_N)$. Hence, $u\ge 0$ in $B_N$, $u\equiv 0$ in $\R^n\setminus B_N$ and $\|u\|_{L^2(B_N)}=1$. Morevoer, by Fatou's Lemma, we have
\begin{align*}
\liminf_{k\to\infty}\iint_{\{u_k>0\}\times\{u_k>0\}} \frac{(u_k(x)-u_k(y))^2}{|x-y|^{n+2\sigma}} \,\ud x\,\ud y\ge \iint_{\{u>0\}\times\{u>0\}} \frac{(u(x)-u(y))^2}{|x-y|^{n+2\sigma}} \,\ud x\,\ud y
\end{align*}
and
\[
\liminf_{k\to\infty} |\{u_k>0\}|\ge |\{u>0\}|.
\]
Thus, $m(N)$ is attained by $u$.
\end{proof}

Now we can show that the variational problem \eqref{prob:variational} is equivalent to \eqref{eq:quantity}.
\begin{prop}\label{prop:equivalentproblems}
We have
\begin{align*}
&\inf\{\lda_{1,\sigma}(\Omega)+|\Omega|: \ \Omega\subset\R^n \mbox{ a bounded open set}\}\\
&= \inf\left\{\frac{I_{n,\sigma,\{u>0\}}[u] }{\|u\|^2_{L^2(\R^n)}}+|\{u>0\}|: u\in \mathring H^\sigma(\R^n),\ \ u\not\equiv 0, u\ge 0\mbox{ in }\R^n\right\}.
\end{align*}
\end{prop}
\begin{proof}
First, it is clear from Lemma \ref{lem:equiv} that 
\begin{align*}
\inf&\{\lda_{1,\sigma}(\Omega)+|\Omega|: \ \Omega\subset\R^n \mbox{ a bounded open set}\}\\
&=\inf_{N\in\mathbb{N}} \big(\inf\left\{\lda_{1,\sigma}(\Omega)+|\Omega|: \Omega\subset B_N \mbox{ is an open set}  \right\}\big)\\
&=\inf_{N \in \mathbb{N}} m(N),
\end{align*}
while from the definition of $m(N)$,
\[
\inf_{N\in\mathbb{N}} m(N)\ge \inf\left\{\frac{I_{n,\sigma,\{u>0\}}[u] }{\|u\|^2_{L^2(\R^n)}}+|\{u>0\}|: u\in \mathring H^\sigma(\R^n),\ \ u\not\equiv 0, u\ge 0\mbox{ in }\R^n\right\}.
\]
Second, given $u\in \mathring H^\sigma(\R^n)$, it is elementary to check that $\eta_N u \in \mathring H^\sigma(\R^n)$ and $\eta_N u \to u$ in $\mathring H^\sigma(\R^n)$ as $N\to\infty$, where $\eta_N(x)=\eta(x/N)$ and  $\eta$ is a standard radial cut-off function supported in $B_2$ and equal to $1$ in $B_1$. Hence, 
\[
\inf_{N\in\mathbb{N}} m(N)\le \inf\left\{\frac{I_{n,\sigma,\{u>0\}}[u] }{\|u\|^2_{L^2(\R^n)}}+|\{u>0\}|: u\in \mathring H^\sigma(\R^n),\ \ u\not\equiv 0, u\ge 0\mbox{ in }\R^n\right\}.
\]
This gives 
\begin{equation}\label{eq:infinumequal}
\inf_{N\in\mathbb{N}} m(N)= \inf\left\{\frac{I_{n,\sigma,\{u>0\}}[u] }{\|u\|^2_{L^2(\R^n)}}+|\{u>0\}|: u\in \mathring H^\sigma(\R^n),\ \ u\not\equiv 0, u\ge 0\mbox{ in }\R^n\right\},
\end{equation}
which implies the conclusion.
\end{proof}

\section{An upper bound}\label{sec:upperbound}
Now we turn our attention to the regularity of minimizing functions $u$ for $m(N)$. At present, we only know that they exist and are nonnegative functions in the space $H^\sigma(\R^n)$, supported on $B_N$. In this section, we will show they are bounded and admit a kind of local maximum principle.

\begin{lem}\label{lem:minimizereq}
Let $u\in \mathring H^\sigma(\R^n)$ be a nonnegative minimizer of $m(N)$. Then it  satisfies
\begin{equation}\label{eq:minimizereq}
(-\Delta)^\sigma_{\{u>0\}} u\le \lambda u\quad\mbox{in }\R^n
\end{equation}
in the weak sense, where
\begin{equation}\label{eq:eigenvalue}
\lambda= \|u\|^{-2}_{L^2(\{u>0\})}\iint_{\{u>0\}\times \{u>0\}}\frac{(u(x)-u(y))^2}{|x-y|^{n+2\sigma}}\,\ud x\ud y.
\end{equation}
More precisely,
\begin{equation}\label{eq:minimizereqd}
\iint_{\{u>0\}\times \{u>0\}}\frac{(u(x)-u(y))(\varphi(x)-\varphi(y))}{|x-y|^{n+2\sigma}}\,\ud x\ud y\le \lambda\int_{\R^n}u(x)\varphi(x)\,\ud x
\end{equation}
for every nonnegative  function $\varphi\in\mathring H^\sigma(\R^n)$.
\end{lem}
\begin{proof}
Let $\varphi\in\mathring H^\sigma(\R^n)$  be a nonnegative  function and $u_t=(u-t\varphi)^+$ for any small $t>0$. Then we have
\[
\frac{I_{n,\sigma,\{u>0\}}[u] }{\|u\|^2_{L^2(B_N)}}+|\{u>0\}| \le \frac{I_{n,\sigma,\{u_t>0\}}[u_t] }{\|u_t\|^2_{L^2(B_N)}}+|\{u_t>0\}|.
\]
Since $\{u_t>0\}\subset \{u>0\}$, we have
\be\label{eq:energyine1}
\frac{I_{n,\sigma,\{u>0\}}[u] }{\|u\|^2_{L^2(B_N)}} \le \frac{I_{n,\sigma,\{u_t>0\}}[u_t] }{\|u_t\|^2_{L^2(B_N)}}\le \frac{I_{n,\sigma,\{u>0\}}[u_t] }{\|u_t\|^2_{L^2(B_N)}}.
\ee
Since $u=(u-t\varphi)^+-(u-t\varphi)^-+t\varphi$, we obtain
\begin{align*}
&(u(x)-u(y))^2\\
&=(u_t(x)-u_t(y))^2 + \Big((u(x)-t\varphi(x))^--(u(y)-t\varphi(y))^-\Big)^2\\
&\quad+2 (u(y)-t\varphi(y))^+(u(x)-t\varphi(x))^-+2 (u(x)-t\varphi(x))^+(u(y)-t\varphi(y))^-\\
&\quad + 2t(\varphi(x)-\varphi(y)) (u(x)-u(y)) - t^2(\varphi(x)-\varphi(y))^2\\
&\ge (u_t(x)-u_t(y))^2+ 2t[\varphi(x)-\varphi(y)] [u(x)-u(y)] - t^2(\varphi(x)-\varphi(y))^2
\end{align*}
and
\begin{align*}
u^2&= u_t^2 + [(u-t\varphi)^-]^2+t^2\varphi^2 + 2t\varphi(u-t\varphi)\\
&\le  u_t^2 + t^2\varphi^2+t^2\varphi^2 + 2t\varphi(u-t\varphi)\\
&= u_t^2 +2t\varphi u.
\end{align*}
Applying these two inequalities in \eqref{eq:energyine1}, we obtain
\begin{align*}
&\|u\|^2_{L^2(B_N)} \iint_{\{u>0\}\times \{u>0\}}\frac{2t[\varphi(x)-\varphi(y)] [u(x)-u(y)] - t^2(\varphi(x)-\varphi(y))^2}{|x-y|^{n+2\sigma}}\,\ud x\ud y\\
&\le2t \int_{B_N}\varphi(x) u(x)\,\ud x \iint_{\{u>0\}\times \{u>0\}}\frac{(u(x)-u(y))^2}{|x-y|^{n+2\sigma}}\,\ud x\ud y.
\end{align*}
By canceling $t$ and sending $t\to 0$, we obtain
\[
\iint_{\{u>0\}\times \{u>0\}}\frac{(u(x)-u(y))(\varphi(x)-\varphi(y))}{|x-y|^{n+2\sigma}}\,\ud x\ud y\le \lambda\int_{B_N}u(x)\varphi(x)\,\ud x
\]
with $\lambda$ as in \eqref{eq:eigenvalue}.
\end{proof}

Next is a Cacciopoli inequality, based on $u$ being a subsolution of this nonlocal equation. Set
\begin{equation}\label{eq:blinearform0}
	\mathscr{B}_u[f, g] = \iint_{\{u > 0\}\times \{u > 0\}} \frac{(f(x) - f(y))(g(x) - g(y))}{|x - y|^{n + 2\sigma}}\,\ud x\ud y.
\end{equation}

\begin{lem}\label{lem:caccipolli1}
Let $0<r<\ell<R$, $\eta: \R^n \rightarrow [0, 1]$ be a smooth radial cutoff function which is $1$ on $B_{r}$, vanishes outside of $B_{\ell}$, and satisfies $|\nabla \eta| \leq 2/(\ell-r)$. Then
	\begin{align*}
		&\iint_{\R^n\times\R^n} \frac{ |u(x)\eta(x) - u(y)\eta(y)|^2}{|x - y|^{n + 2\sigma}} \\
		&\leq C \int_{B_{\ell}} u(x)\,\ud x \int_{\R^n\setminus B_{R}} \frac{u(y)}{|y-x|^{n + 2\sigma}} \,\ud y+ C \int_{B_{R}} u^2 +\frac{CR^{2-2\sigma}}{(R-r)^{2}} \int_{B_{R}} u^2
	\end{align*}
\end{lem}

\begin{proof}
 We use $\eta^2 u$ as a test function, to get that
	\[
		\mathscr{B}_u[u, \eta^2 u] \leq C \int_{\R^n} u^2 \eta^2 \leq C \int_{B_{\ell}} u^2.
	\]
	The quantity on the left may be subdivided as
	\[
		\mathscr{B}_u[u, \eta^2 u] = \left(\int_{(B_{R} \cap \{u > 0\})^2} + 2 \int_{B_{R} \cap \{u > 0\} \times \{u > 0\} \setminus B_{R}}\right)\frac{[u(x) - u(y)][\eta^2 u(x) - \eta^2 u(y)]}{|x - y|^{n + 2\sigma}}.
	\]
	
	The integral over the second region can be estimated using the fact that the support of $\eta$ is a distance $R-\ell$ from the complement of $B_{R}$:
	\begin{align*}
		&- \iint_{B_{R} \cap \{u > 0\} \times \{u > 0\} \setminus B_{R}} \frac{[u(x) - u(y)][\eta^2 u(x) - \eta^2 u(y)]}{|x - y|^{n + 2\sigma}} \\
		&\quad \leq C \int_{B_{R}}\eta^2 u(x)\,\ud x \int_{\{u > 0\} \setminus B_{R}} \frac{u(y)}{|x - y|^{n + 2\sigma}}\,\ud y\\
		&\quad\le C \int_{B_{\ell}} u(x)\,\ud x \int_{\R^n\setminus B_{R}} \frac{u(y)}{|y-x|^{n + 2\sigma}} \,\ud y.
	\end{align*}
	
	Let us focus on the much more difficult estimate on the other region. We claim that,
	\begin{align*}
		& |u(x) - u(y)|^2 (\eta^2(x) + \eta^2(y)) \\
		& \leq 4 [ (u(x) - u(y))(\eta^2 u(x)  - \eta^2 u(y)) + (u(x) + u(y))^2 (\eta(x) - \eta(y))^2].
	\end{align*}
	To see this, start with the more obvious identity
	\[
		|u(x) - u(y)|^2 \eta^2(x) = (u(x) - u(y))(\eta^2(x)u(x) - \eta^2(y) u(y)) - u(y)(u(x) - u(y))(\eta^2(x) - \eta^2(y)).
	\]
	Now reverse the role of $x$ and $y$, add, factor, and apply Cauchy's inequality to the last term:
	\begin{align*}
		&|u(x) - u(y)|^2 (\eta^2(x) + \eta^2(y)) \\& = 2 (u(x) - u(y))(\eta^2(x)u(x) - \eta^2(y) u(y)) - (u(y) + u(x))(u(x) - u(y))(\eta^2(x) - \eta^2(y))\\
		& \leq 2 (u(x) - u(y))(\eta^2(x)u(x) - \eta^2(y) u(y)) \\
		&\quad+ |u(x) + u(y)|^2 |\eta(x) - \eta(y)|^2 + \frac{1}{4} |u(x) - u(y)|^2 |\eta(x) + \eta(y)|^2\\
		& \leq 2 (u(x) - u(y))(\eta^2(x)u(x) - \eta^2(y) u(y)) \\
		&\quad+ 2 |u^2(x) + u^2(y)| |\eta(x) - \eta(y)|^2 + \frac{1}{2} |u(x) - u(y)|^2 |\eta^2(x) + \eta^2(y)|.
	\end{align*}
	Reabsorb the rightmost term and multiply by $2$ to recover the claimed inequality.
	
	Also, from Cauchy's inequality, we have
	\[
		|u(x)\eta(x) - u(y)\eta(y)|^2 \leq 2|u(x) - u(y)|^2 \eta^2(x) +2 |\eta(x) - \eta(y)|^2 u^2(y).
	\]
	Switching the role of $x$ and $y$, we obtain
	\[
		|u(x)\eta(x) - u(y)\eta(y)|^2 \leq |u(x) - u(y)|^2 (\eta^2(x)+\eta^2(y)) + |\eta(x) - \eta(y)|^2 (u^2(x)+u^2(y)).
	\]		
Therefore, we obtain
	\begin{align*}
		& |u(x)\eta(x) - u(y)\eta(y)|^2\\
		& \leq 5 [ (u(x) - u(y))(\eta^2 u(x)  - \eta^2 u(y)) + 2(u^2(x) + u^2(y)) (\eta(x) - \eta(y))^2].
	\end{align*}
		
	Using this, we immediately get that
	\begin{align*}
		&\iint_{(\{u>0\} \cap B_R) ^2} \frac{ |u(x)\eta(x) - u(y)\eta(y)|^2}{|x - y|^{n + 2\sigma}} \\
		&\leq 5 \iint_{(\{u>0\} \cap B_R) ^2}  \frac{(u(x) - u(y))(\eta^2 u(x) - \eta^2 u(y))}{|x - y|^{n + 2\sigma}} \\
		&\quad+ C (R-r)^{-2} \int_{\{u>0\} \cap B_R} u^2(x) \int_{\{u>0\} \cap B_R} |x - y|^{2 - n -2\sigma} dy dx.
	\end{align*}
	For the second term on the right, we have
	\[
		C (R-r)^{-2} \int_{\{u>0\} \cap B_R} u^2(x) \int_{\{u>0\} \cap B_R} |x - y|^{2 - n -2\sigma} dy dx \leq \frac{CR^{2-2\sigma}}{(R-r)^{2}} \int_{B_{R}} u^2.
	\]

	Combining with our earlier estimates, we obtain
	\begin{align*}
		&\iint_{(\{u>0\} \cap B_R) ^2} \frac{ |u(x)\eta(x) - u(y)\eta(y)|^2}{|x - y|^{n + 2\sigma}} \\
		&\leq C \int_{B_{\ell}} u(x)\,\ud x \int_{\R^n\setminus B_{R}} \frac{u(y)}{|y-x|^{n + 2\sigma}} \,\ud y+ C \int_{B_{R}} u^2 +\frac{CR^{2-2\sigma}}{(R-r)^{2}} \int_{B_{R}} u^2.
	\end{align*}
Since $\eta u$ is supported $\{u>0\} \cap B_R$, then the conclusion follows from the Hardy inequality in Lemma \ref{lem:equivalentnorms} (and also Lemma \ref{lem:sobolevforus}).
\end{proof}

We can now obtain a local bound for minimizers, similar to Theorem 1.1 of Di Castro - Kuusi - Palatucci \cite{CKP}. Their work concerns (sub)solutions to the \emph{full} fractional Laplacian $(- \Delta)^\sigma_{\R^n}$. In our setting, $(-\Delta)^\sigma_{\R^n} u > (-\Delta)^\sigma_{\{u > 0\}} u$, so these results can not applied directly.

\begin{lem}\label{lem:upperbound}
Let $0<R\le 1$. Then there exists $C=C(n,\sigma)$ such that 
\[
\sup_{B_{R/2}} u\le \va R^{2\sigma} \int_{\R^n} \frac{u(y)}{(R+|y|)^{n + 2\sigma}} \,\ud y+C\va^{-\frac{n}{4\sigma}} R^{-\frac{n}{2}}\|u\|_{L^2(B_R)}
\]
for all $\va\in (0,1]$.
\end{lem}

\begin{proof}
Let $L>0$,  $0<r<\ell<R$, $\eta: \R^n \rightarrow [0, 1]$ be a radial cutoff function which is $1$ on $B_{r}$, vanishes outside of $B_{\ell}$, and satisfies $|\nabla \eta| \leq 2/(\ell-r)$. As in the proof of Lemma \ref{lem:caccipolli1}, we use $\eta^2(u-L)^+$ as the test function. Then we have
\[
\mathscr{B}_u[u,\eta^2(u-L)^+] \le \int_{\R^n} \eta^2 u (u-L)^+\le \int_{\R^n} \eta^2  [(u-L)^+]^2+L\int_{\R^n} \eta^2  (u-L)^+.
\]
Using
\[
u(x)-u(y)=(u(x)-L)^+-(u(y)-L)^--(u(x)-L)^-+(u(y)-L)^-,
\]
one has
\[
\mathscr{B}_u[(u-L)^+,\eta^2(u-L)^+] \le \mathscr{B}_u[u,\eta^2(u-L)^+].
\]
Following from the proof of Lemma \ref{lem:caccipolli1}, one has
\begin{align*}
\|\eta (u-L)^+\|_{\mathring H^\sigma(\R^n)}^2 
&\leq   C \int_{B_{\ell}} (u(x)-L)^+\,\ud x \int_{\R^n\setminus B_{R}} \frac{(u(y)-L)^+}{|y-x|^{n + 2\sigma}} \,\ud y + C \int_{B_{R}} [(u-L)^+]^2 \\
&\quad+\frac{CR^{2-2\sigma}}{(R-r)^{2}} \int_{B_{R}}  [(u-L)^+]^2+ C L\int_{B_{R}} [(u-L)^+].
\end{align*}
By the Sobolev inequality, we have
\begin{align*}
\|(u-L)^+\|_{L^{\frac{2n}{n-2\sigma}}(B_r)}^2 
&\leq   C \int_{B_{\ell}} (u(x)-L)^+\,\ud x \int_{\R^n\setminus B_{R}} \frac{(u(y)-L)^+}{|y-x|^{n + 2\sigma}} \,\ud y + C \int_{B_{R}} [(u-L)^+]^2 \\
&\quad+\frac{CR^{2-2\sigma}}{(R-r)^{2}} \int_{B_{R}}  [(u-L)^+]^2+ C L\int_{B_{R}} [(u-L)^+].
\end{align*}
For $k=0,1,2,\cdots$, let $L_k=(1-2^{-k})M$, $R_k=\frac{1}{2}(1+2^{-k})R$, where $M>0$ to be fixed in the end,  $\ell_k=\frac{1}{2}(R_k+R_{k+1})$,  and
\[
U_k=\|(u-L_k)^+\|_{L^{2}(B_{R_k})}. 
\]
We have
\begin{align*}
U_{k+1}&\le \|(u-L_{k+1})^+\|_{L^{\frac{2n}{n-2\sigma}}(B_{R_{k+1}})} |\{u>L_{k+1}\}\cap B_{R_{k+1}}|^{\frac{\sigma}{n}}\\
&\le \|(u-L_{k+1})^+\|_{L^{\frac{2n}{n-2\sigma}}(B_{R_{k+1}})} (U_k 2^{k+1}/M)^\frac{2\sigma}{n}.
\end{align*}
Also,
\begin{align*}
\int_{B_{R_k}} [(u-L_{k+1})^+]^2&\le U_k^2,\\
 L_{k+1}\int_{B_{R_k}} [(u-L_{k+1})^+]&\le \frac{M}{L_{k+1}-L_k} \int_{B_{R_k}} [(u-L_{k})^+]^2\le 2^{k+1}U_k^2,
\end{align*}
and
\begin{align*}
& \int_{B_{\ell_k}} (u(x)-L_{k+1})^+\,\ud x \int_{\R^n\setminus B_{R_k}} \frac{(u(y)-L_{k+1})^+}{|y-x|^{n + 2\sigma}} \,\ud y\\
 &\le \frac{2^{k+1}}{M} U_k^2 \ 2^{(k+3)(n+2\sigma)} \int_{\R^n} \frac{u(y)}{(R+|y|)^{n + 2\sigma}} \,\ud y.
\end{align*}
Denote 
\[
A=R^{2\sigma} \int_{\R^n} \frac{u(y)}{(R+|y|)^{n + 2\sigma}} \,\ud y,\quad\mbox{and}\quad \widetilde U_k=R^{-\frac{n}{2}}U_k.
\]
Then we obtain 
\begin{align*}
\widetilde U_{k+1}
&\leq   (C^k  \widetilde U_k)^{1+\frac{2\sigma}{n}}\left(\sqrt{\frac{A}{M}}+1\right)M^{-2\sigma/n}.
\end{align*}
Taking
\[
M\ge \va A, 
\]
then
\begin{align*}
\widetilde U_{k+1}
&\leq   (C^k  \widetilde U_k)^{1+\frac{2\sigma}{n}}\va^{-1/2}M^{-2\sigma/n}.
\end{align*}
Hence $\widetilde U_{k}\to 0$ as $k\to\infty$ provided that
\[
M\ge C\va^{-\frac{n}{4\sigma}}  \widetilde U_0.
\]
Hence, we can take 
\[
M= \va A+C\va^{-\frac{n}{4\sigma}}\widetilde U_0,
\]
and thus, 
\[
\sup_{B_{R/2}} u\le \va R^{2\sigma} \int_{\R^n} \frac{u(y)}{(R+|y|)^{n + 2\sigma}} \,\ud y+C\va^{-\frac{n}{4\sigma}} R^{-\frac{n}{2}}\|u\|_{L^2(B_R)}.
\]
\end{proof}

\begin{rem}
The $\varepsilon$ in Lemma \ref{lem:upperbound} interpolates between the local and nonlocal terms. It plays an essential role when we prove a lower bound for $u$ in Section \ref{sec:lowerbound}.
\end{rem}

Consequently, we have

\begin{thm}\label{thm:Linfinitybound}
There exists a positive constant $C(n,\sigma)$ such that
\[
\|u\|_{L^\infty(\R^n)}\le C(n,\sigma) \|u\|_{L^2(\R^n)}.
\]
\end{thm}
\begin{proof}
This follows from Lemma \ref{lem:upperbound} and the H\"older inequality.
\end{proof}

\begin{lem}\label{lem:minimizereq2}
Let $u\in \mathring H^\sigma(\R^n)$ be a nonnegative minimizer of $m(N)$. Then it satisfies
\begin{equation}\label{eq:minimizereq2}
(-\Delta)^\sigma_{\{u>0\}} u\ge \lambda u\quad\mbox{in } \{u>0\}
\end{equation}
in the weak sense, with $\lambda$ as in \eqref{eq:eigenvalue}. In other words,
\begin{equation}\label{eq:minimizereqd2}
\iint_{\{u>0\}\times \{u>0\}}\frac{(u(x)-u(y))(\varphi(x)-\varphi(y))}{|x-y|^{n+2\sigma}}\,\ud x\ud y\ge \lambda\int_{\R^n}u(x)\varphi(x)\,\ud x
\end{equation}
for every nonnegative function $\varphi\in \mathring H^\sigma(\R^n)$ such that $|\supp \varphi \setminus \{u > 0\}| = 0$. 
\end{lem}
\begin{proof}
Let $\varphi\in \mathring H^\sigma(\R^n)$ such that $|\supp \varphi \setminus \{u > 0\}| = 0$,  and $u_t=u+t\varphi$ for any small $t>0$. Then we have
\[
\frac{I_{n,\sigma,\{u>0\}}[u] }{\|u\|^2_{L^2(B_N)}}+|\{u>0\}| \le \frac{I_{n,\sigma,\{u_t>0\}}[u_t] }{\|u_t\|^2_{L^2(B_N)}}+|\{u_t>0\}|.
\]
Since $\{u_t>0\}= \{u>0\}$, we have
\be\label{eq:energyine2}
\frac{I_{n,\sigma,\{u>0\}}[u] }{\|u\|^2_{L^2(B_N)}}\le \frac{I_{n,\sigma,\{u>0\}}[u_t] }{\|u_t\|^2_{L^2(B_N)}}.
\ee
Then
\begin{align*}
&\int_{\{u>0\}} (2tu\varphi+t^2\varphi^2)\cdot  \iint_{\{u>0\}\times \{u>0\}}\frac{(u(x)-u(y))^2}{|x-y|^{n+2\sigma}}\,\ud x\ud y\\
&\le \|u\|^2_{L^2(B_N)} \iint_{\{u>0\}\times \{u>0\}}\frac{2t[\varphi(x)-\varphi(y)] [u(x)-u(y)] + t^2(\varphi(x)-\varphi(y))^2}{|x-y|^{n+2\sigma}}\,\ud x\ud y.
\end{align*}
By canceling $t$ and sending $t\to 0$, we obtain
\[
\iint_{\{u>0\}\times \{u>0\}}\frac{(u(x)-u(y))(\varphi(x)-\varphi(y))}{|x-y|^{n+2\sigma}}\,\ud x\ud y\ge \lambda\int_{B_N}u(x)\varphi(x)\,\ud x
\]
with $\lambda$ as in \eqref{eq:eigenvalue}.
\end{proof}

The following corollary follows from Lemma \ref{lem:minimizereq} and Lemma \ref{lem:minimizereq2}, making precise the fact that a minimizer $u$ solves the eigenvalue equation on its domain of positivity.

\begin{cor}\label{cor: equation}
Let $u\in \mathring H^\sigma(\R^n)$ be a nonnegative minimizer of $m(N)$.  Let $\varphi\in \mathring H^\sigma(\R^n)$ such $|\{ \varphi \neq 0 \} \setminus \{ u > 0 \}| = 0$. Then
\begin{equation}\label{eq:minimizerequal}
\iint_{\{u>0\}\times \{u>0\}}\frac{(u(x)-u(y))(\varphi(x)-\varphi(y))}{|x-y|^{n+2\sigma}}\,\ud x\ud y= \lambda\int_{\R^n}u(x)\varphi(x)\,\ud x
\end{equation}
where $\lambda$ as in \eqref{eq:eigenvalue}.
\end{cor}

\section{A bound from below and existence}\label{sec:lowerbound}

We also can derive a lower bound for minimizers. The precise statement of it below takes on an unusual nonlocal form: it guarantees the existence of some (large) scale $R$, possibly depending on the point considered, where $\sup_{B_R(x)} u$ is at least $R^{\sigma}$. We note that this is not a ``uniform'' estimate (i.e. true \emph{for all} $R$), nor do we expect a uniform estimate of this form to be valid: we believe that the exponent $\sigma$ here is actually not the correct rate of growth for $u$ near $\partial \{u > 0\}$. Nonetheless, this estimate will be enough for concentration compactness arguments in proving existence below, while an optimal, uniform estimate appears to require new ideas.

\begin{lem}\label{lem:lowerbound}
	Let $u$ be a nonnegative minimizer (normalized so that $\int_{\R^n} u^2 = 1$) of $m(N)$ defined in \eqref{eq:equivalentvp}. Then there are small numbers $c_*, R_*$ (depending only on $n$ and $\sigma$) such that if $x$ is a Lebesgue point of $\{u > 0\}$, then for some $R \in [R_*, 1]$,
	\[
		\sup_{B_{R}(x)} u \geq c_* R^{\sigma}.
	\]
\end{lem}

\begin{proof}
	After a translation we may assume $x = 0$. Take $\eta: \R^n \rightarrow [0, 1]$ to be a radial cutoff function with $\eta =1$ outside $B_{R/2}$, $\eta = 0$ precisely on $B_{R/10}$, and $|\nabla \eta|\leq 4/R$. We will use $w = \eta u$ as a competitor for $u$. Note that $\{w > 0\} \subset \{u > 0\}$.
	
	First, reducing the domain of positivity decreased the associated Gagliardo norm:
	\[
		\mathscr{B}_w[w, w] \leq \mathscr{B}_u[w, w],
	\]
	where $\mathscr{B}_{\cdot}[\cdot,\cdot]$ is defined in \eqref{eq:blinearform0}. We may rewrite the right as
	\[
		\mathscr{B}_u[w, w] = \mathscr{B}_u[u, u] - 2 \mathscr{B}_u[v, u] + \mathscr{B}_u[v, v],
	\]
	where $v = u - w = (1- \eta)u$.	Now apply Lemma \ref{lem:minimizereq2} with $\phi = v$ to give
	\[
		\mathscr{B}_u[v, u] \geq \lambda \int_{\R^n} v u \geq 0,
	\]
	and thus,
	\[
		\mathscr{B}_w[w, w] \leq \mathscr{B}_u[u, u] + \mathscr{B}_u[v, v].
	\]
	It follows from Lemma \ref{lem:caccipolli1} that 
	\[
		\mathscr{B}_u[v, v] \leq C \int_{B_{R}} u(x)\,\ud x \int_{\R^n} \frac{u(y)}{(R + |y|)^{n + 2\sigma}} \,\ud y+ C R^{-2\sigma} \int_{B_{R}} u^2.
	\]
	
	Now from the minimality of $u$, we get
	\begin{align*}
		\mathscr{B}_u[u, u] + |\{u > 0\}|&\leq \frac{\mathscr{B}_w[w, w]}{\int w^2} + |\{w > 0\}| \\
		&\leq \frac{\mathscr{B}_u[u, u] + C\int_{B_{R}} u \int_{\R^n} \frac{u(y)}{(R + |y|)^{n + 2\sigma}}\,\ud y + C R^{-2\sigma} \int_{B_{R}} u^2}{1 - 2 \int_{\R^n} u v } \\
		&\quad+ |\{u > 0\} \setminus B_{R/10}|.
	\end{align*}
	Since $u$ is bounded (see Theorem \ref{thm:Linfinitybound}), and  $\int_{\R^n} uv \le \int_{B_{R/2}} u^2\le CR^2$, which can be small if $R\le R_0$ for some small $R_0$, we obtain
	\[
		|\{u > 0\} \cap B_{R/10}| \leq C \int_{B_{R}} u(x)\,\ud x \int_{\R^n} \frac{u(y)}{(R + |y|)^{n + 2\sigma}} \,\ud y+ C R^{-2\sigma} \int_{B_{R}} u^2.
	\]

	We now argue as follows. Assume that for all $S \in [R, R']$, where $R\le (R')^2\le 1$, we have that $\sup _{B_S} |u| \leq c_0 S^{\sigma}$, where $c_0 \leq c_*$ is small. We will show that this guarantees that
	\[
		\sup_{B_{R/20}} u \leq \frac{1}{2} c_0 \left(\frac{R}{20}\right)^\sigma +  \left(\frac{R}{20}\right)^{2\sigma} (R')^{- 2\sigma}.
	\]
	Indeed, we have that in this case,
	\begin{equation}
		\int_{B_R}u^2 \leq c_0^2 R^{n + 2\sigma} ,\quad  \int_{B_{R}} u(x)\,\ud x \int_{\R^n} \frac{u(y)}{(R + |y|)^{n + 2\sigma}} \,\ud y \leq C c_0 R^{n}[c_0 +R^\sigma(R')^{-2\sigma}], 
	\end{equation}
	so 
	\begin{equation}\label{e:i}
		|\{u > 0\} \cap B_{R/10}| \leq C c_0 R^n.
	\end{equation}
	Applying Lemma \ref{lem:upperbound},
	\begin{align}
		\sup_{B_{R/20}} u& \leq \va R^{2\sigma} \int_{\R^n} \frac{u(y)}{(R+|y|)^{n + 2\sigma}} \,\ud y+C\va^{-\frac{n}{4\sigma}} R^{-\frac{n}{2}}\|u\|_{L^2(B_{R/10})}\nonumber\\
		&\leq C \va R^{2\sigma}(c_0R^{-\sigma}+ (R')^{- 2\sigma}) + C\va^{-\frac{n}{4\sigma}} c_0^{3/2} R^{\sigma}\nonumber\\
		&= C\va c_0 R^\sigma+C\va R^{2\sigma}(R')^{-2\sigma}+  C\va^{-\frac{n}{4\sigma}} c_0^{3/2} R^{\sigma}.
	\end{align}
By choosing $\va=1/(2C(20)^\sigma)$, and $c_*$  sufficiently small, this gives the promised estimate.
	
	 Let $4  R_0^\sigma \le c_*<1/20$. Now assume that for $S \in [R_0, 1]$, $\sup_{B_S} u \leq c_* S^{\sigma}$.  We argue that this is a contradiction. We may apply our claim with $c_0 = c_*$ for every $R$ in $[R_0, 20R_0]$ (with $R' = 1$) to give that
	\begin{equation}
		\sup_{B_{S}} u \leq \frac{1}{2} c_* S^{\sigma} +  S^{2\sigma} \leq \frac{1}{2} c_* S^{\sigma} + R_0^\sigma S^{\sigma}  \leq \frac{3}{4}c_* S^{\sigma} \quad\mbox{for all } S \in [R_0/20, R_0].
	\end{equation}
 We may continue applying the claim with $c_0 = c_*$ and $R \in [R_0/20^k, R_0/20^{k-1}]$ with $R'=1$ to get, inductively, that this holds on $R \in [R_0/20^{k+1}, R_0/20^{k}]$. Hence, we have
 	\begin{equation}
		\sup_{B_{S}} u \leq \frac{1}{2} c_* S^{\sigma} +  S^{2\sigma} \leq \frac{3}{4}c_* S^{\sigma} \quad\mbox{for all } S \le R_0.
	\end{equation}
	
	Next, we may slightly improve this estimate. To do so, we set $R \leq 20 R_1 \ll R_0$ and $R' = R_0$, and apply the claim with $c_0 = \frac{3}{4}c_*$ instead. This gives
	\[
		\sup_{B_{S}} u \leq \frac{3}{8} c_* S^{\sigma} + S^{2\sigma}R_0^{-2\sigma} \leq [\frac{3}{8} c_*  + (R_1/R_0^2)^\sigma]S^{\sigma}\quad\mbox{for any }S \leq R_1.
	\]
If $(R_1/R_0^2)^\sigma = \frac{1}{4}\cdot \frac{3}{4}c_*$, this gives that
	\[
	\sup_{B_{S}} u \leq (\frac{3}{4})^2 c_* S^{\sigma}\quad\mbox{for all } S\le R_1.
	\]
 We may continue on in this fashion, choosing $(R_k/R_{k - 1}^2)^\sigma = \frac{1}{4} (\frac{3}{4})^k c_*$, to get that if for 
 \[
 \sup_{B_S} u \leq (3/4)^k c_* S^\sigma\quad\mbox{for all }S \leq R_{k-1},
 \]
then
	\[
		\sup_{B_S} u \leq [\frac{1}{2} (3/4)^{k} c_* + (R_k/R_{k-1}^2)^\sigma] S^\sigma \leq (3/4)^{k + 1} S^\sigma c_*\quad\mbox{for all }S\le R_k.
	\]

	Combining with the intermediate estimate \eqref{e:i}, it gives
	\[
		|\{u > 0\} \cap B_{R_k/10}| \leq C (\frac34)^k c_* R_k^n.
	\]
	This gives that the Lebesgue density of $\{u > 0\}$ at $0$ is $0$, which contradicts the assumptions.
\end{proof}

\begin{cor}\label{cor:L2below}
Let $u$ be as in Lemma \ref{lem:lowerbound}. There exists $c=c(n,\sigma)$ such that
\[
\int_{B_2(x)} u^2 \geq c.
\]
for every  Lebesgue point $x$ of $\{u > 0\}$
\end{cor}
\begin{proof}
Let $R$ and $R_*$ be the ones in Lemma \ref{lem:lowerbound}. Assume $x=0$. Combining the results in Lemma \ref{lem:upperbound} and Lemma \ref{lem:lowerbound}, we have
\begin{align*}
R^\sigma & \le \va R^{2\sigma} \int_{\R^n} \frac{u(y)}{(R+|y|)^{n + 2\sigma}} \,\ud y+C\va^{-\frac{n}{4\sigma}} R^{-\frac{n}{2}}\|u\|_{L^2(B_{2R})}\\
&\le C \va + C\va^{-\frac{n}{4\sigma}} R^{-\frac{n}{2}}\|u\|_{L^2(B_{2R})}.
\end{align*}
By choosing $\va=R^\sigma/(2C)$, we have
\[
\|u\|_{L^2(B_{2R})}\ge CR^{\frac{3n}{4}+\sigma}\ge CR_*^{\frac{3n}{4}+\sigma}.
\]
This finishes the proof.
\end{proof}

This leads to:

\begin{lem}\label{lem:boundeddiameter} 
Let $u$ be as in Lemma \ref{lem:lowerbound}.  Then there is a large number $K = K(n ,\sigma)$ and another nonnegative $v\in  \mathring H^\sigma(\R^n)$ with $\int_{\R^n} v^2 = 1$ such that $v$ is supported on $B_K$ and
	\[
		\mathscr{B}_v[v, v] + |\{v > 0\}|\leq \mathscr{B}_u[u, u] + |\{ u > 0\}|.
	\]
\end{lem}

\begin{proof}
	Select a representative of $u$ which vanishes except at Lebesgue points of $\{u > 0\}$. By the Besicovitch covering theorem, from the collection of balls $\{B_2(x) : x\in \{u> 0\} \}$, we can select a finite-overlapping subcover of $\{u > 0\}$, $U=\{B_2(x_i)\}_{i \in I}$. Together with Corollary \ref{cor:L2below}, we have from above that
	\[
		\int_{B_2(x_i)} u^2 \geq c,
	\]
	but
	\[
		\sum_i \int_{B_2(x_i)} u^2 \leq C \int_{\{ u > 0 \}} u^2 = C.
	\]
	Let $M = \#(I)$ be the number of balls in $U$: this gives that $M < \infty$ and bounded in terms of $n$ and $\sigma$ only. 
	
	Assume that $u$ has the following property: $u = \sum_i u_i$, where $u_i$ are nonzero and $|\{u_i > 0\} \cap \{u_j > 0\}| = 0$ for all $i\neq j$. Then we see that
	\[
	\sum_i \mathscr{B}_{u_i}[u_i, u_i] =\sum_{i} \mathscr{B}_{u_i}[u, u] < \mathscr{B}_u[u, u],
	\]
	as all the cross terms are positive. Also
	\[
	\sum_{i} \int_{\R^n} u_i^2 = 1 \qquad \sum_{i} |\{u_i > 0\}| = |\{u > 0\}|.
	\]
	Set $v_i(x) = R^{-n/2}_i u_i(R_ix)(\int_{\R^n} u_i^2)^{-1/2}$ (this has $\int_{\R^n} v_i^2 = 1$ ) and compute the energy:
	\[
		\mathscr{B}_{v_i}[v_i, v_i] + |\{v_i > 0\}| = R_i^{2s} \frac{\mathscr{B}_{u_i}[u_i, u_i]}{\int_{\R^n} u_i^2} + R_i^{-n} |\{u_i > 0\}|.
	\]
	Choose $R_i = (|\{u_i > 0\}|/|\{u > 0\}|)^{1/n}$, noting that $R_i^{2s} < 1$. Then we have that
	\[
		\sum_i \int_{\R^n} u_i^2 [\mathscr{B}_{v_i}[v_i, v_i] + |\{v_i > 0\}|] = |\{u > 0\}| + \sum_i R_i^{2s} \mathscr{B}_{u_i}[u_i, u_i] < \mathscr{B}_u[u, u] + |\{u > 0\}|.
	\]
	Since $\sum_i \int_{\R^n} u_i^2 = 1$, then at least one of the $v_i$ has to have
	\[
		\mathscr{B}_{v_i}[v_i, v_i] + |\{v_i > 0\}| < \mathscr{B}_u[u, u] + |\{u > 0\}|.
	\]
	
	Let $\{U_j\}_{j\in J}$ be the connected components of the set $\cup_{i\in I} B_2(x_i)$. If there is only one connected component, then the function $v(x) = u(x - x_1)$ is supported on $B_{M + 1}$ and we may conclude. If not, let $u_j = u|_{U_j}$ and apply the above construction to find one $v_i$ with energy less than $u$. Note that as $|\{ u_i > 0 \}| \geq c \int_{U_1}u^2 \geq c$, we have that $R_i \geq c$, so a translate of $v_i$ will be supported on $B_{(M+1)/c}$. We use this $v_i$ to conclude.
\end{proof}

Now we are ready to prove Theorem \ref{thm:existenceg}.

\begin{proof}[Proof of Theorem \ref{thm:existenceg}]
Let $K(n,\sigma)$ be the one in Lemma \ref{lem:boundeddiameter}, and $N_1,N_2 >K(n,\sigma)$. Let $u_1$ and $u_2$ be minimizers of $m(N_1)$ and $m(N_2)$. Due to the translation invariance of $m(N)$, $u_1$ is a valid competitor of $u_2$ for $m(N_2)$, and $u_2$ is a valid competitor of $u_1$ for $m(N_1)$. Therefore, $m(N_1)=m(N_2)$. Because of \eqref{eq:infinumequal}, $u_1$ is a desired minimzier. 
\end{proof}

\section{Discussion and an open question}\label{sec:discussion}
We know from Theorem \ref{thm:existenceg} that the minimizing set is bounded with bounded support. A natural question is whether the minimizing eigenfunction $u_0$ is continuous (and in particular, whether $\{u_0 > 0\}$ admits an open representative); this was stated as Open Question \ref{ques:continuity}. We do not know how to prove this, and in this section, we would like to discuss the difficulties and present some related open questions.

\begin{lem}\label{lem:harmonic}
Let $\Omega\subset\R^n$ be a fixed set such that $B_{2r}(x_0)\subset\Omega$. Let $v\in\mathring H^\sigma(\R^n)$ be weak solution of $(-\Delta)^\sigma_\Omega v=0$ in $B_r(x_0)$. Assume that $|v|\le 1$ on $\R^n$. Then for every $\alpha\in(0,1)$, there exists $C>0$ depending only on $n, \alpha$ and $\sigma$, such that
\[
[v]_{C^\alpha(B_{r/2}(x_0))}\le C r^{-\alpha}. 
\]
\end{lem}
\begin{proof}
We first assume that $r=1$ and $x_0=0$. Since $B_2\subset\Omega$, we know that $\int_{\R^n\setminus\Omega}\frac{1}{|x-y|^{n+2\sigma}}\,\ud y$ is a smooth function in $B_{3/2}$ and all its derivative are bounded in $B_{3/2}$ independent of $\Omega$. Then it follows from standard estimates for the fractional Laplacian in the whole space that for every $\alpha\in(0,1)$,
\[
[v]_{C^\alpha(B_{1/2})}\le C,
\]
where $C$ depends only on $n,\alpha$ and $\sigma$.

In general, we let $w(x)=v(x_0+rx)$. Then
\[
(-\Delta)^\sigma_{\widetilde\Omega} w=0\quad \mbox{in } B_1,
\]
where $\widetilde\Omega=(\Omega-x)/r$ containing $B_2$. Then we have
\[
[w]_{C^\alpha(B_{1/2})}\le C.
\]
Rescaling back, we have
\[
[v]_{C^\alpha(B_{r/2}(x_0))}\le C r^{-\alpha}.
\]
\end{proof}

Let $u_0$ be a minimizer in Theorem \ref{thm:existenceg}. Let us normalize it so that $\|u_0\|_{L^2(\R^n)}=1$, and hence, $u_0\le C$ in $\R^n$. Let $x_0\in\R^n,R>0$, and 
\[
\Omega=\{u_0>0\}\cup B_{2R}(x_0).
\]
Let $h\in\mathring H(\R^n)$ be the solution of $(-\Delta)^\sigma_{\Omega} h=0$ in $B_R(x_0)$ and $h\equiv u_0$ in $\R^n\setminus B_R(x_0)$.  Since $u_0\ge 0$, we have that $h\ge 0$ in $\Omega$. Since $u_0\le C$, we have $h\le C$. 

Moreover,
\begin{align}
\mathscr{B}_\Omega[u_0-h,u_0-h]&=\mathscr{B}_\Omega[u_0-h,u_0+h]-2\mathscr{B}_\Omega[u_0-h,h]\nonumber\\
&=\mathscr{B}_\Omega[u_0,u_0]-\mathscr{B}_\Omega[h,h]\nonumber\\
&\le \mathscr{B}_\Omega[u_0,u_0]-\mathscr{B}_{\{h>0\}}[h,h]\label{eq:bilineaadhoc},
\end{align}
where 
\[
\mathscr{B}_\Omega[v,\varphi]=\iint_{\Omega\times\Omega}\frac{(v(x)-v(y))(\varphi(x)-\varphi(y))}{|x-y|^{n+2\sigma}}\,\ud x\ud y,
\]
and we used the equation of $h$ in the second inequality so that $\mathscr{B}_\Omega[u_0-h,h]=0$. On the other hand, using $h$ as a competitor for $u_0$, we know that
\be\label{eq:competitor}
\mathscr{B}_{\{u_0>0\}}[u_0,u_0] + |\{u_0>0\}| \le \frac{\mathscr{B}_{\{h>0\}}[h,h]}{\|h\|_{L^2(\Omega)}} + |\Omega|.
\ee
Since
\[
\int_{\R^n} |h^2-u_0^2|\le CR^n, 
\]
we have
\[
\int_{\R^n} h^2\ge1- CR^n.
\]
From $\mathscr{B}_\Omega[u_0-h,h]=0$, we have $\mathscr{B}_\Omega[h,h]\le \mathscr{B}_\Omega[u_0,u_0]\le \mathscr{B}_{\R^n}[u_0,u_0]\le C$. Then, from \eqref{eq:competitor}, we obtain
\[
\mathscr{B}_{\{u_0>0\}}[u_0,u_0] \le \mathscr{B}_{\{h>0\}}[h,h]+CR^n.
\]
Therefore, from \eqref{eq:bilineaadhoc}, we obtain
\[
\mathscr{B}_\Omega[u_0-h,u_0-h]\le \mathscr{B}_\Omega[u_0,u_0]-\mathscr{B}_{\{u_0>0\}}[u_0,u_0]+CR^n.
\]
Using the Poincar\'e inequality, we have 
\[
\int_{\R^n}|u_0-h|^2= \int_{B_R(x_0)}|u_0-h|^2 \le C R^{2\sigma}\mathscr{B}_{B_R(x_0)}[u_0-h,u_0-h].
\]
Hence, 
\begin{align*}
\int_{\R^n}|u_0-h|^2 &\le  C R^{2\sigma} (\mathscr{B}_\Omega[u_0,u_0]-\mathscr{B}_{\{u_0>0\}}[u_0,u_0])+CR^{n+2\sigma}\\
&=  C R^{2\sigma} \int_{\{u_0>0\}}u_0(x)^2 \int_{B_{2R}(x_0)\setminus\{u_0>0\}}\frac{\ud y}{|x-y|^{n+2\sigma}}\,\ud x+CR^{n+2\sigma}.
\end{align*}

\begin{ques}
Is it true that
\begin{align*}
\int_{\{u_0>0\}}u_0(x)^2 \int_{B_{2R}(x_0)\setminus\{u_0>0\}}\frac{\ud y}{|x-y|^{n+2\sigma}}\,\ud x\le CR^{N}
\end{align*}
for some $N>n-2\sigma$?
\end{ques}

This is the difficulty of the above attempt of proving the minimizer's global continuity. We do not know how to obtain such a sufficient decay estimate. If the solution to this open question is true, then it will lead to an affirmative answer  to Open Question \ref{ques:continuity}.




	The boundary behavior of solutions to homogeneous Dirichlet problems for the regional fractional Laplace equation (on sufficiently regular domains) was studied by  Chen-Kim-Song \cite{ChKiSo}, culminating in two-sided heat kernel estimates (and consequently, the two-sided Green's function estimates for the regional fractional Laplacian in Proposition 4.2 of \cite{FJX}). These results show that solutions on $\Omega$ grow like $d^{2s - 1}(\cdot, \partial \Omega)$ from the boundary. It is natural, then, to ask whether such growth estimates can be established for our minimizers $u_0$ near $\pa\{u_0>0\}$:

\begin{ques} \label{q:3}
Does there exist $C>0$ such that 
\begin{align*}
u_0(x)\le C [\mbox{dist}(x,\pa\{u_0>0\})]^{2\sigma-1}\quad \forall \,x\in\{u_0>0\}?
\end{align*}
\end{ques}

	The point here is that we do not know that $\partial \{u_0 > 0\}$ is regular, so this is does not follow directly from \cite{ChKiSo}. A related question is whether there is a complementary uniform \emph{lower} bound on the growth of solutions:

\begin{ques}\label{q:4}
Does there exist $c>0$ such that for all $R\in (0,1)$ and all $x\in\pa\{u_0>0\}$, 
\begin{align*}
\sup_{B_R(x)} u_0 \geq c R^{2\sigma-1}?
\end{align*}
\end{ques}

	This form of a lower bound is common in the free boundary literature. Note that the scaling here, $R^{2s - 1}$, is different from what was used in Lemma \ref{lem:lowerbound} to obtain a non-uniform lower bound: indeed, this difference in scaling explains why the estimate there could not be uniform in $R$. It also suggests new ideas, and possibly new competitor sets and functions, are needed to answer Open Questions \ref{q:3} and \ref{q:4}, as most of the arguments in this paper do not detect this $R^{2s-1}$ scaling.

	The Euler-Lagrange equation for the minimizer is another interesting open question. We have already seen that $u_0$ satisfies 
$(-\Delta)_{\{u_0>0\}}^\sigma u_0=\lambda u_0$ in $\{u_0>0\}$, but by analogy to the second order case \cite{AC} we expect another pointwise equation to be satisfied by $u_0$ along the boundary $\partial \{u_0 > 0\}$, perhaps of the following form:

\begin{ques}
Does $\pa\{u_0>0\}$ have some approximate normal vectors $\nu_x$ in an appropriate sense, and does \begin{align*}
\lim_{t \rightarrow 0} \frac{u_0(x - t \nu_x)}{t^\alpha} = const?
\end{align*}
\end{ques}

Finally, let us suggest the following question about the shape of minimizers. Unlike the previous sequence of questions, a positive answer cannot be attained by local regularity or competitor arguments, and requires some new global approach.

\begin{ques}
	Are minimizers $u_0$ radial and unique up to translations?
\end{ques}

\bigskip

\noindent T. Jin

\noindent Department of Mathematics, The Hong Kong University of Science and Technology\\
Clear Water Bay, Kowloon, Hong Kong\\[2mm]
 \textsf{Email: tianlingjin@ust.hk}

\bigskip

\noindent D. Kriventsov

\noindent Department of Mathematics, Rutgers University\\
110 Frelinghuysen Road, Piscataway, NJ 08854, USA \\[2mm]
 \textsf{Email: dnk34@math.rutgers.edu }

\bigskip

\noindent J. Xiong

\noindent School of Mathematical Sciences, Laboratory of Mathematics and Complex Systems, MOE\\
Beijing Normal University, Beijing 100875, China\\[2mm]
 \textsf{Email: jx@bnu.edu.cn}
 

\begin{thebibliography}{99}

\addtolength\itemsep{-0.3mm}

\bibitem{AC}  H. W. Alt,and  L. A. Caffarelli, \textit{Existence and regularity for a minimum problem with free boundary.} J. Reine Angew. Math. \textbf{325} (1981), 105–144.

\bibitem{AL}F.J. Almgren Jr. and E.H. Lieb,
\textit{ Symmetric decreasing rearrangement is sometimes continuous}. 
\newblock {\em J. Amer. Math. Soc.} 2(4): 683--773, 1989.







\bibitem{BBC} K. Bogdan, K. Burdzy and Z.-Q. Chen, \textit{Censored stable processes}. Probab. Theory Relat. Fields \textbf{127} (2003), no. (1), 89--152.







\bibitem{BrNi} H. Br\'ezis and L. Nirenberg, \textit{Positive solutions of nonlinear elliptic equations involving critical Sobolev exponents}. Comm. Pure Appl. Math. \textbf{36} (1983), no. 4, 437--477.




\bibitem{CKP} A. Di Castro, T. Kuusi and G. Palatucci,
\emph{Local behavior of fractional p-minimizers.}
Ann. Inst. H. Poincar\'e Anal. Non Lin\'eaire \textbf{33} (2016), no. 5, 1279--1299.

\bibitem{ChKiSo} Z.-Q. Chen, P. Kim, R. Song, \textit{Two-sided heat kernel estimates for censored stable-like processes}. Probab. Theory Related Fields \textbf{146} (2010), no. 3-4, 361--399. 



\bibitem{Dy} B. Dyda, \textit{A fractional order Hardy inequality}. Illinois J. Math. \textbf{48} (2004), no. 2, 575--588. 

\bibitem{DyFr} B. Dyda and R. L. Frank, \textit{Fractional Hardy--Sobolev--Maz’ya inequality for domains}. Studia Math. \textbf{208} (2012), no. 2, 151--166.


\bibitem{Faber} G. Faber, 
\textit{Beweis, dass unter allen homogenen Membranen von gleicher Fl\"ache und gleicher Spannung die kreisf\"ormige den tiefsten Grundton gibt}. 
Sitzungsber. Bayer. Akad. Wiss. M\"unchen, Math. Phys. Kl. (1923) pp. 169--172


\bibitem{FJX} R. L. Frank, T. Jin and J. Xiong, \textit{Minimizers for the fractional Sobolev inequality on domains}, Calc. Var. Partial Differential Equations \textbf{57} (2018), no. 2, Art. 43, 31 pp.












\bibitem{Guan} Q.-Y. Guan, \textit{Integration by parts formula for regional fractional Laplacian}. Commun. Math. Phys. \textbf{266} (2006), no.2, 289--329.

\bibitem{GuanMa} Q.-Y. Guan and Z.-M. Ma, \textit{Reflected symmetric $\alpha$-stable processes and regional fractional Laplacian}. Probab. Theory Relat. Fields \textbf{134} (2006), no.4, 649--694.


\bibitem{Krahn} E. Krahn,
\textit{\"Uber eine von Rayleigh formulierte Minimaleigenschaft des Kreises}.
Math. Ann. \textbf{94} (1925), no. 1, 97--100. 



\bibitem{LW} D. Li and K. Wang, \textit{Symmetric radial decreasing rearrangement can increase the fractional Gagliardo norm in domains}, Commun. Contemp. Math. \textbf{21} (2019), no. 7, 1850059, 9 pp..

\bibitem{Lieb} E. H. Lieb, \textit{Sharp constants in the Hardy--Littlewood--Sobolev and related inequalities}. Ann. of Math. (2) \textbf{118} (1983), 349--374.

\bibitem{LiebLoss} E. H. Lieb and M. Loss, Analysis. Second edition. Graduate Studies in Mathematics, 14. American Mathematical Society, Providence, RI, 2001. xxii+346 pp.

\bibitem{Lin} F.-H. Lin, \textit{Extremum problems of Laplacian eigenvalues and generalized Polya conjecture}. Chin. Ann. Math. Ser. B \textbf{38} (2017), no. 2, 497--512.



\bibitem{LS}
M.~Loss and C.~Sloane.
\textit{Hardy inequalities for fractional integrals on general domains}.
J. Funct. Anal., \textbf{259} (2010), no. 6, 1369--1379.

\bibitem{MN} R. Musina and A.I. Nazarov, 
\textit{Sobolev inequalities for fractional Neumann Laplacians on half spaces.}
Adv. Calc. Var. \textbf{14} (2021), no. 1, 127–145.



\bibitem{PS}
G. P\'olya and G. Szeg\"o.
\newblock Isoperimetric Inequalities in Mathematical Physics. 
\newblock {\em Annals of Mathematics Studies}, no. 27, Princeton University Press, Princeton, N. J., 1951.

\bibitem{Rayleigh} J.W.S. Rayleigh, 
\newblock The Theory of Sound. 
2d ed. Dover Publications, New York, N. Y., 1945. Vol. I, xlii+480 pp.; vol. II, xii+504 pp. 


\bibitem{SVV} Y. Sire, J.L. V\'azquez and B. Volzone,
\textit{Symmetrization for fractional elliptic and parabolic equations and an isoperimetric application.}
Chin. Ann. Math. Ser. B \textbf{38} (2017), no. 2, 661--686. 



\end{thebibliography}
\end{document}